\documentclass[11pt,a4paper]{article}
\usepackage[a4paper,margin=2.35cm]{geometry}
\usepackage[T1]{fontenc}
\usepackage{lmodern}
\usepackage{amsmath,amssymb,amsthm,mathtools}
\usepackage{booktabs,longtable,array,enumitem,multirow}
\usepackage{xcolor,hyperref,fancyhdr,microtype}
\usepackage{tikz}
\usetikzlibrary{intersections}
\usepackage[nameinlink,noabbrev]{cleveref}
\usepackage{url}

\definecolor{linkblue}{RGB}{25,86,150}
\hypersetup{colorlinks=true,linkcolor=linkblue,urlcolor=linkblue,citecolor=linkblue}
\setlist{itemsep=2pt,topsep=4pt}
\newcolumntype{P}[1]{>{\raggedright\arraybackslash}p{#1}}
\newcolumntype{C}[1]{>{\centering\arraybackslash}p{#1}}

\newtheorem{theorem}{Theorem}[section]
\newtheorem{lemma}[theorem]{Lemma}

\newtheorem{corollary}[theorem]{Corollary}
\theoremstyle{definition}
\newtheorem{definition}[theorem]{Definition}
\newtheorem{remark}[theorem]{Remark}

\newcommand{\file}[1]{\path{#1}}

\title{Circles Determined by Planar Point Sets\\
\large Erd\H{o}s Problem 506}
\author{Liyan Wang}
\date{2026.08.18}

\begin{document}
\maketitle

\begin{abstract}
For $n\ge4$, let $c(n)$ be the minimum number of distinct circles containing at least
three points of an $n$-point set in the Euclidean plane, where the set is neither
collinear nor concyclic.  Put
\[
F(n)=1+\binom{n-1}{2}-\left\lfloor\frac{n-1}{2}\right\rfloor.
\]
We determine $c(n)$ for every $n\ge4$: it equals $F(n)$ apart from three
exceptional orders.  We also solve the variant in which no three points are
collinear; that variant has a single exceptional order.  The proofs and exact
finite verifications were developed through a collaboration between human
researchers and artificial-intelligence systems.
\end{abstract}


\section{Introduction}

\begin{definition}
Let $P\subset\mathbb R^2$ be finite, and let $c(P)$ denote the number of distinct
Euclidean circles containing at least three points of $P$.  For $n\ge4$, define
\[
c(n)=\min_P c(P),
\]
where the minimum is over all $n$-point sets $P$ that are neither collinear nor
concyclic.  We call a finite point set satisfying these two conditions
\emph{admissible}.  A line is not counted as a circle, and a circle containing
$k\ge3$ points is counted once.
\end{definition}

Take $n-1$ points on a circle $\Gamma$ and a point $p$ outside $\Gamma$.  The points
on $\Gamma$ can be grouped into $\lfloor(n-1)/2\rfloor$ pairs such that each pair is
collinear with $p$.  Apart from $\Gamma$, every other pair of points on $\Gamma$
that is not collinear with $p$ determines a circle with $p$.  Distinct pairs give
distinct circles, because two distinct circles have at most two common points.
Thus
\begin{equation}\label{eq:F-en}
c(n)\le F(n):=1+\binom{n-1}{2}-\left\lfloor\frac{n-1}{2}\right\rfloor.
\end{equation}

\begin{figure}[ht]
\centering
\begin{tikzpicture}[scale=0.82,every node/.style={font=\small}]
  \coordinate (p) at (3.2,0);
  \path[name path=gamma] (0,0) circle (2);
  \draw[thick,linkblue] (0,0) circle (2);
  \foreach \y in {-1.45,-0.48,0.48,1.45}{
    \path[name path=secant] (p)--(-2.7,\y);
    \path[name intersections={of=gamma and secant,name=z}];
    \draw[gray] (p)--(-2.7,\y);
    \fill (z-1) circle (2.1pt);
    \fill (z-2) circle (2.1pt);
  }
  \fill[linkblue] (p) circle (2.4pt) node[right] {$p$};
  \node[linkblue] at (0,2.35) {$\Gamma$};
\end{tikzpicture}
\caption{The benchmark construction, illustrated for $n=9$.  Four secants through
$p$ pair the eight points of $\Gamma$; those four collinear pairs do not produce
circles through $p$.}
\label{fig:benchmark-en}
\end{figure}
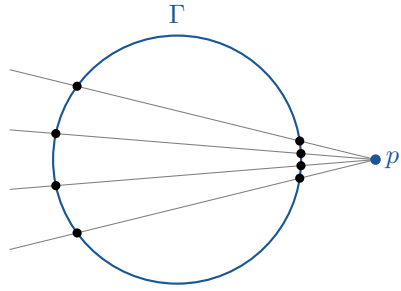

Elliott initiated a systematic study of this problem in 1967.  His Theorem~1 proves
that every admissible set of $n>3$ points determines at least
\[
\frac{2n(n-1)}{63}
\]
circles containing exactly three points; such circles are usually called
\emph{ordinary circles}.  His Theorem~2 asserts that the total number of determined
circles is at least $\binom{n-1}{2}$ for $n>393$
\cite[Theorems~1--2]{elliott-en}.  Purdy and Smith identified the missing
subtraction caused by the pairs made collinear with the exterior point in
\cref{fig:benchmark-en}.  They observed that Elliott's argument can be modified,
without changing the threshold, to give $c(n)\ge F(n)$ for $n\ge394$; together
with \eqref{eq:F-en}, this yields $c(n)=F(n)$ in that range
\cite[\S2.1]{purdy-smith-en}.

B\'alintov\'a and B\'alint proved the general bound
\[
c(P)\ge \frac{15n(n-1)+1678}{266}\qquad(n\ge6)
\]
in their Theorem~2.4, using the pointed ordinary-circle estimate of their
Theorem~2.1 and the six-point initial case in Lemma~6.1
\cite[Theorems~2.1 and 2.4 and Lemma~6.1]{balintova-en}.  A related but distinct
line of work counts only ordinary circles.  Zhang's Theorem~4.1 gives at least
$m/6$ ordinary lines avoiding a prescribed external point; inversion and double
counting then yield the ordinary-circle bound $\frac19\binom n2$ in his
Theorem~4.2 \cite[Theorems~4.1--4.2]{zhang-en}.  Lin et al. later determined the
exact minimum number of ordinary circles for sufficiently large $n$
\cite[Theorem~1.1(i)]{lin-et-al-en}.  These results provide useful tools, but they
do not directly settle the present extremal problem: a circle containing many
points is counted once here, although it contains many triples.

The results certified in this paper are as follows.

\begin{theorem}\label{thm:main-en}
For every $n\ge4$, one has $c(n)=F(n)$, with precisely three exceptions:
\[
c(6)=8=F(6)-1,\qquad c(7)=11=F(7)-2,\qquad
c(8)=17=F(8)-2.
\]
\end{theorem}

Elliott records Segre's observation that a planar projection of a cube gives an
eight-point counterexample to Elliott's proposed bound
\cite[p.~182, discussion following Theorem~2]{elliott-en}.  That configuration
determines 18 circles.  We prove that the first exceptional order already occurs
at $n=6$, and we improve the eight-point construction to one with $c(P)=17$.
Explicit constructions for the exceptional cases are shown in
\cref{fig:c6-en,fig:c7-en,fig:c8-en}.

Let $c_{\mathrm{nc}}(n)$ denote the minimum of $c(P)$ over admissible
$n$-point sets having no three collinear points.  The additional hypothesis
changes the answer only once.

\begin{corollary}\label{thm:no-three-en}
For every $n\ge4$,
\[
 c_{\mathrm{nc}}(n)=1+\binom{n-1}{2},
\]
except that $c_{\mathrm{nc}}(8)=20$.
\end{corollary}

Purdy and Smith's correction of Elliott's argument settles every $n\ge394$ and
therefore makes the unresolved part finite in principle
\cite[\S2.1]{purdy-smith-en}.  In practice, however, the remaining incidence
structures form a very large and irregular search space, for which a direct
verification is neither small nor naturally uniform.  The present proof is
independent of that large-$n$ threshold: it establishes the result for every
$n\ge15$, rather than merely filling the interval $15\le n\le393$.  We also solve
the no-three-collinear variant.  For the original problem with $n\le14$, exact
computer verification is applied only after mathematical constraints have reduced
the possible incidence data.

The project was developed through collaboration between human researchers and
artificial-intelligence systems.  Eureka, an AI-assisted mathematical research
system developed by the JiuChong team at the University of Science and Technology
of China, supplied the initial numerical verification for $n\ge16$.  After a
human proof for $n\ge92$ had been found, OpenAI Codex helped sharpen it to the
uniform range $n\ge15$ and carried out the exact small-order verifications and
searches for exceptional configurations.  The mathematical part of the argument
has also been translated into Lean~4 and checked against Mathlib.  The formalization
and exact-verification files are available in the accompanying GitHub repository
\url{https://github.com/LyonWang00/Erdos-problem/tree/main/Erdos506}.

\section{Notation and basic tools}

Fix an admissible $n$-point set $P$.  Let $K$ be the largest number of points of
$P$ on a line or a circle, and put $r=n-K$.  A line or circle attaining $K$ will be
called a \emph{largest line} or \emph{largest circle}, respectively.
Every maximal collinear or concyclic subset of $P$ containing at least three
points will be called, respectively, a \emph{line block} or a \emph{circle block};
when the distinction is immaterial, either is called a \emph{block}.

Whenever the points are labelled $p_0,p_1,\ldots$, a string of indices denotes
the corresponding subset; for example, $0256$ means
$\{p_0,p_2,p_5,p_6\}$.  In an incidence list the surrounding text specifies
whether the subset is collinear or concyclic.  This convention is used only to
shorten finite lists; it does not identify a circle with an index string unless
at least three of the indicated points are noncollinear.

\subsection{Counting from a largest collinear or concyclic subset}

\begin{lemma}\label{lem:largest-en}
If a largest line contains $K$ points of $P$, then
\begin{equation}\label{eq:DL-en}
c(P)\ge D_L(n,r):=r\binom K2-\binom r2\left\lfloor\frac K2\right\rfloor.
\end{equation}
If a largest circle contains $K$ points of $P$, then
\begin{equation}\label{eq:DC-en}
c(P)\ge D_C(n,r):=1+r\left(\binom K2-\left\lfloor\frac K2\right\rfloor\right)
-\binom r2\left\lfloor\frac K2\right\rfloor.
\end{equation}
\end{lemma}

\begin{proof}
Fix a point $q$ outside the chosen largest line or circle.  In the line case, $q$
and any pair on the line determine a circle, and these circles are distinct for
fixed $q$.  In the circle case, at most $\lfloor K/2\rfloor$ pairs on the circle
are collinear with $q$; every other pair gives a new circle, in addition to the
largest circle itself.  Fix two exterior points $q,q'$.  Every circle counted for
both points contains $q$ and $q'$ and meets the chosen largest line or largest circle
in a two-point set.  The two-point sets belonging to two distinct repeated circles
are disjoint: otherwise the circles would share $q,q'$ and a third point and hence
coincide.  Thus at most $\lfloor K/2\rfloor$ circles are common to the two lists.
The first Bonferroni inequality (sum of list sizes minus the sum of pairwise
intersections) now gives both displayed lower bounds; a circle occurring in three or
more lists does not invalidate this lower estimate.
\end{proof}

\subsection{Inversion and multiplicity identities}

Fix $x\in P$ and invert $P\setminus\{x\}$ about $x$, obtaining a noncollinear
configuration $Q_x$ of $q=n-1$ points.  Indeed, if $Q_x$ were collinear, its inverse
would be either a line through $x$ or a circle through $x$, so all of $P$ would be
collinear or concyclic.  Let $t_i(x)$ be the number of connecting
lines of $Q_x$ that contain exactly $i$ points.  Let $b_i(x)$ be the number of these
lines that pass through the original geometric point $x$, and set
\[
u_i(x)=t_i(x)-b_i(x).
\]
A connecting line through the inversion centre corresponds to a line in the original
configuration.  An $i$-point connecting line avoiding the centre corresponds to a
circle through $x$ containing exactly $i+1$ original points.  This gives the following
identities.

\begin{lemma}\label{lem:identities-en}
For every $x\in P$,
\begin{equation}\label{eq:pairs-en}
\sum_{i=2}^{K-1}\binom i2t_i(x)=\binom{n-1}{2}
\end{equation}
and
\begin{equation}\label{eq:center-en}
\sum_{i=2}^{K-1}i b_i(x)\le n-1
\end{equation}
For every integer $s\ge2$, one also has
\begin{equation}\label{eq:rounded-centre-en}
\sum_{i\ge s}b_i(x)\le
\left\lfloor\frac{n-1}{s}\right\rfloor.
\end{equation}
If $\ell_k$ and $c_k$ are respectively the numbers of $k$-point lines and
$k$-point circles in $P$, then
\begin{equation}\label{eq:circle-en}
c(P)=\sum_{x\in P}\sum_{i=2}^{K-1}\frac{u_i(x)}{i+1}
\end{equation}
Moreover,
\begin{equation}\label{eq:incidence-en}
\sum_xb_i(x)=(i+1)\ell_{i+1},\qquad
\sum_xu_i(x)=(i+1)c_{i+1}
\end{equation}
and
\begin{equation}\label{eq:triples-en}
\binom n3=\sum_{k\ge3}\binom k3(\ell_k+c_k)
\end{equation}
In particular, $\sum_xt_i(x)$ is divisible by $i+1$.
\end{lemma}

\begin{proof}
Equation~\eqref{eq:pairs-en} partitions the pairs of $Q_x$ by their connecting line.
Distinct connecting lines through $x$ contain disjoint subsets of $Q_x$, which gives
\eqref{eq:center-en}.  Each line counted on the left of
\eqref{eq:rounded-centre-en} uses at least $s$ points of these disjoint subsets;
the number of such lines is an integer, which proves the rounded inequality.  A
$k$-point circle becomes, under inversion at each one of its points, a $(k-1)$-point
line avoiding the centre.  It is therefore counted exactly $k$ times in
\eqref{eq:circle-en}; the same argument for lines gives
\eqref{eq:incidence-en}.  Finally, every noncollinear triple lies on a unique circle,
whereas every collinear triple lies on a unique line, which proves
\eqref{eq:triples-en}.
\end{proof}

\subsection{Line-arrangement inequalities and intersection constraints}

Elliott's equation~(4) is
\begin{equation}\label{eq:melchior-en}
t_2(x)\ge3+\sum_{i\ge4}(i-3)t_i(x),
\end{equation}
which follows from the Euler relation for a real projective line arrangement
\cite[Eq.~(4)]{elliott-en}.  When the three highest intersection multiplicities
vanish, Hirzebruch's equation~(9) gives
\begin{equation}\label{eq:hirz-en}
4t_2(x)+3t_3(x)\ge4(n-1)+4\sum_{i\ge5}(2i-9)t_i(x).
\end{equation}
Its original hypotheses are $t_q=t_{q-1}=t_{q-2}=0$; whenever we use the inequality,
these hypotheses are verified from the current upper bound on $K$
\cite[\S6, Eq.~(9)]{hirzebruch-en}.  Under the corresponding multiplicity
hypotheses, the Bojanowski inequality is
\begin{equation}\label{eq:bojanowski-en}
4t_2(x)+3t_3(x)\ge4q+\sum_{i\ge5}i(i-4)t_i(x),
\qquad q=n-1,
\end{equation}
provided no connecting line contains at least $2q/3$ image points, as stated in
the dual form in Pokora's Theorem~2.1~\cite[Theorem~2.1]{pokora-en}.

\begin{lemma}[Rounded lower bound for the number of connecting lines]
\label{lem:rounded-lines-en}
Whenever \eqref{eq:melchior-en} and \eqref{eq:bojanowski-en} apply,
\begin{equation}\label{eq:rounded-lines-en}
 \sum_{i\ge2}t_i(x)\ge
 \left\lceil\frac{q^2+3q+9}{9}\right\rceil .
\end{equation}
\end{lemma}

\begin{proof}
Add $2/9$ of the pair identity \eqref{eq:pairs-en}, $1/3$ of
\eqref{eq:melchior-en}, and $1/9$ of \eqref{eq:bojanowski-en}.  The coefficient
of $t_i$ is
\[
 \frac{2}{9}\binom i2-\frac{i-3}{3}-\frac{i(i-4)}9=1
 \quad(i\ge5),
\]
and direct substitution gives the same coefficient for $i=2,3,4$.  The
right side is $(q^2+3q+9)/9$.  The left side is integral, which gives the
ceiling.
\end{proof}

Apply Zhang's Theorem~4.1 to the noncollinear set $Q_x$ and to the external point
$x$ (the centre is not an image point)~\cite[Theorem~4.1]{zhang-en}.  It gives
\begin{equation}\label{eq:zhang-en}
u_2(x)\ge\left\lceil\frac{n-1}{6}\right\rceil,
\end{equation}
and the theorem of Csima and Sawyer, for $n>7$, gives
$t_2(x)\ge\lceil6(n-1)/13\rceil$
\cite[the theorem on pp.~187--188]{csima-sawyer-en}.

Order the connecting lines of $Q_x$ containing at least four points as
$L_1,L_2,\ldots$, with nonincreasing sizes.  Every initial segment satisfies
\begin{equation}\label{eq:prefix-en}
\sum_{a=1}^j|L_a|-\binom j2\le n-1,
\end{equation}
because two distinct lines share at most one point.  If a fixed connecting line
contains $a$ points, the $a(n-1-a)$ pairs between this line and its complement must
belong to other connecting lines.  If two fixed lines contain $a$ and $b$ points,
then after deleting their possible intersection there remain $(a-1)(b-1)$ cross
pairs.  Hence
\begin{equation}\label{eq:oneline-en}
\sum_i(i-1)t_i-(a-1)\ge a(n-1-a)
\end{equation}
and
\begin{equation}\label{eq:twoline-en}
\sum_it_i-2\ge(a-1)(b-1)
\end{equation}
These are necessary rather than sufficient conditions for geometric realizability,
but they eliminate many impossible integer multiplicity vectors before coordinates
enter the problem.

We shall also need a constraint involving several inversion centres at once.  Let
$\mathcal C_k$ be a specified family of $k$-point circles, let
$g_k=|\mathcal C_k|$, and let $d_k(x)$ be the number of circles in $\mathcal C_k$
through $x$.  Since two distinct circles meet in at most two points,
\begin{equation}\label{eq:same-en}
\sum_x\binom{d_k(x)}2\le2\binom{g_k}{2}
\end{equation}
and, for $k\ne m$,
\begin{equation}\label{eq:cross-en}
\sum_xd_k(x)d_m(x)\le2g_kg_m
\end{equation}

\begin{lemma}[A second-moment estimate for circle incidences]
\label{lem:moment-en}
Let $\mathcal C_k$ and $\mathcal C_m$ consist of $g_k$ $k$-point circles and
$g_m$ $m$-point circles, respectively.  Suppose that a local count $\Phi_x$ satisfies
\[
\Phi_x\ge \alpha+\beta d_k(x)+\beta'd_m(x)
-\gamma\binom{d_k(x)}2-\gamma'\binom{d_m(x)}2
-\eta d_k(x)d_m(x),
\]
where $\gamma,\gamma',\eta\ge0$.  Then
\begin{align}
\sum_{x\in P}\Phi_x\ge{}&n\alpha+\beta k g_k+\beta'mg_m
-2\gamma\binom{g_k}{2}-2\gamma'\binom{g_m}{2}
-2\eta g_kg_m. \label{eq:moment-en}
\end{align}
\end{lemma}

\begin{proof}
Double counting gives $\sum_xd_k(x)=kg_k$ and $\sum_xd_m(x)=mg_m$.
Sum the local inequality over $x$ and apply \eqref{eq:same-en}--\eqref{eq:cross-en}
to the terms with nonpositive coefficients.
\end{proof}

Thus one need not enumerate the full distribution of incidence degrees over all
centres.  It suffices to find a quadratic lower bound of the displayed form for the
finite table of local minima.  The case of one circle family is obtained by setting
$d_m=0$.

\section{A mathematical proof for \texorpdfstring{$n\ge17$}{n>=17}}

Put \(q=n-1\), and recall that \(K\) is the largest number of original points
on a line or a circle.  Every connecting line in an inverted configuration
\(Q_x\) therefore contains at most \(K-1\) image points.  Write
\[
 \Phi_x=\sum_{i=2}^{K-1}\frac{t_i(x)-b_i(x)}{i+1};
\]
by \eqref{eq:circle-en}, \(c(P)=\sum_x\Phi_x\).

\begin{lemma}\label{lem:bojanowski-local-en}
Suppose that
\begin{equation}\label{eq:bojanowski-range-en}
 3(K-1)<2(n-1).
\end{equation}
Then, for every \(x\in P\),
\begin{equation}\label{eq:bojanowski-local-en}
\Phi_x\ge
 \frac{K+1}{18K}\binom q2+\frac{K+1}{2K}
 +\frac{K-2}{9K}q-\frac13\left\lfloor\frac q2\right\rfloor.
\end{equation}
\end{lemma}

\begin{proof}
Condition \eqref{eq:bojanowski-range-en} is precisely the maximum-multiplicity
hypothesis needed for Bojanowski's inequality
\[
4t_2+3t_3\ge4q+\sum_{i\ge5}i(i-4)t_i
\]
\cite[Theorem~2.1]{pokora-en}.  Multiply the point-pair identity, Melchior's
inequality, and this inequality by
\[
 \alpha=\frac{K+1}{18K},\qquad
 \beta=\frac{K+1}{6K},\qquad
 \gamma=\frac{K-2}{36K},
\]
respectively, and add them.  The resulting coefficient of \(t_i\) is at most
\(1/(i+1)\).  For \(i=2,3\) equality holds; the differences between
\(1/(i+1)\) and the resulting coefficient are
\[
 \frac{K-5}{30K}\quad(i=4),\qquad
 \frac{(i-3)(i-2)(K-i-1)}{12K(i+1)}\quad(5\le i\le K-1).
\]
They are nonnegative.  (When \(K=3,4\), the \(i=4\) term is absent.)  Hence
\[
 \sum_i\frac{t_i(x)}{i+1}\ge
 \alpha\binom q2+3\beta+4q\gamma.
\]
The connecting lines counted by the \(b_i(x)\) pass through the inversion
centre and contain pairwise disjoint subsets of the image points.  Consequently
\[
 \sum_i b_i(x)\le\left\lfloor\frac q2\right\rfloor,\qquad
 \sum_i\frac{b_i(x)}{i+1}\le
 \frac13\left\lfloor\frac q2\right\rfloor,
\]
which proves \eqref{eq:bojanowski-local-en}.
\end{proof}

\begin{theorem}\label{thm:nge17-en}
If \(n\ge17\), every admissible \(n\)-point set \(P\) satisfies
\(c(P)\ge F(n)\).
\end{theorem}

\begin{proof}
First assume \eqref{eq:bojanowski-range-en}.  Summing
\eqref{eq:bojanowski-local-en} over the \(n\) inversion centres and subtracting
\(F(n)-1\) gives, when \(n\) is even,
\begin{equation}\label{eq:even-margin-en}
 \frac{Kn^3-23Kn^2+100Kn-72K+n^3-11n^2+28n}{36K},
\end{equation}
and, when \(n\) is odd,
\begin{equation}\label{eq:odd-margin-en}
 \frac{Kn^3-23Kn^2+94Kn-54K+n^3-11n^2+28n}{36K}.
\end{equation}
Both expressions decrease with \(K\), since their derivative is
\[
 -\frac{n(n-7)(n-4)}{36K^2}.
\]
Condition \eqref{eq:bojanowski-range-en} gives
\(K\le(2n+1)/3\).  At this continuous right endpoint,
\eqref{eq:even-margin-en} and \eqref{eq:odd-margin-en} become, respectively,
\[
 \frac{n^4-21n^3+72n^2+20n-36}{18(2n+1)},\qquad
 \frac{n^4-21n^3+66n^2+35n-27}{18(2n+1)}.
\]
For even \(n=18+2u\) and odd \(n=19+2u\), the two numerators are
\[
\begin{split}
16u^4+408u^3+3528u^2+11056u+6156,\\
16u^4+440u^3+4140u^2+14472u+10746,
\end{split}
\]
so they are positive for \(u\ge0\).  At \(n=17\), integrality gives
\(K\le11\), and \eqref{eq:odd-margin-en} has minimum \(10/33\).
Thus \(c(P)>F(n)-1\), and integrality yields \(c(P)\ge F(n)\).

It remains to consider \(3(K-1)\ge2(n-1)\).  Since \(n\ge17\), this implies
\(3K\ge n+12\).  Put \(r=n-K\).  Using
\(\lfloor K/2\rfloor\le K/2\), the two estimates in
Lemma~\ref{lem:largest-en} give, respectively,
\[
 D_L(n,r)\ge \frac{rK(3K-n-1)}4
     =B(n,K)+\frac{rK}{2}-1\ge B(n,K)
\]
and
\[
 D_C(n,r)\ge1+\frac{rK(3K-n-3)}4=B(n,K),
\]
where
\begin{equation}\label{eq:coarse-en}
 B(n,K):=1+\frac{(n-K)K(3K-n-3)}4.
\end{equation}
On the real interval \((n+12)/3\le K\le n-2\), the derivative of \(B\) is a
concave quadratic and is positive at the left endpoint, where it equals
\[
 \frac{2n^2+21n-360}{12}>0.
\]
Because a concave quadratic can cross the horizontal axis from positive to negative
at most once on this interval, (B) is either increasing throughout or first
increasing and then decreasing.  Hence (B) has no interior minimum.  Using
\(F(n)\le(n^2-4n+6)/2\), its margins at the two endpoints are
\[
 5n-38,\qquad \frac{(n-7)(n-2)}2,
\]
both positive.  For \(K=n-1\), the exact largest-circle expression in
\eqref{eq:DC-en} equals \(F(n)\), while the largest-line expression exceeds it
by \(\lfloor K/2\rfloor-1\).  This closes the complementary branch.
\end{proof}

The supplementary algebra audit expands every coefficient difference and both
parity polynomials with exact symbolic arithmetic.  It does not enumerate
\(n\), \(K\), or any line-multiplicity vector.

\section{The layer \texorpdfstring{$n=16$}{n=16}}
\label{sec:n16-en}

The local bound \eqref{eq:bojanowski-local-en} remains valid at \(n=16\).
For \(K=3,4,5,6\), its local values and the resulting margins above \(98\)
are
\[
\begin{array}{c|cccc}
K&3&4&5&6\\ \hline
 \Phi_x&20/3&77/12&94/15&37/6\\
 16\Phi_x-98&26/3&14/3&34/15&2/3.
\end{array}
\]
For (K=9,\ldots,15), the smaller of the applicable line and circle bounds in
Lemma~\ref{lem:largest-en} is, respectively,
\[
141,166,201,205,199,162,99.
\]
These numbers are all at least (F(16)=99).  It remains to treat (K=7,8),
which we now do without a finite verification.

Suppose first that \(K=7\), and put
\(d_x=t_6(x)-b_6(x)\); this is the number of seven-point circles through
\(x\).  Suppress the centre from the notation and define
\[
\begin{aligned}
 T&=\sum_{i=2}^6\binom i2t_i,&
 M&=t_2-\sum_{i=4}^6(i-3)t_i,\\
 B&=4t_2+3t_3-\sum_{i=5}^6i(i-4)t_i,&
 S&=\sum_{i=2}^6b_i.
\end{aligned}
\]
The pair identity, Melchior, Bojanowski, and the fact that distinct connecting
lines through the inversion centre contain disjoint sets of image points give,
respectively,
\[
 T=105,\qquad M\ge3,\qquad B\ge60,\qquad S\le7.
\]
A direct comparison of coefficients gives the identity
\begin{equation}\label{eq:n16-k7-identity-en}
\begin{split}
 \Phi_x+\frac{d_x}{42}
={}&\frac7{108}T+\frac7{36}M+\frac1{54}B-\frac13S+R_7,\\
 R_7={}&\frac{t_4}{180}+\frac{b_3}{12}+\frac{2b_4}{15}
          +\frac{b_5}{6}+\frac{b_6}{6}\ge0.
\end{split}
\end{equation}
Consequently
\begin{equation}\label{eq:n16-k7-local-en}
 \Phi_x+\frac{d_x}{42}\ge\frac{37}{6}.
\end{equation}

Let \(g\) be the number of seven-point circles.  Four such circles cannot
pass through one point: after inversion their four six-point lines would have
union of size at least \(4\cdot6-\binom42=18>15\).  Thus \(d_x\le3\), and
\[
 \sum_xd_x=7g,\qquad
 \sum_x\binom{d_x}{2}\le2\binom g2=g(g-1).
\]
If \(g=5\) or \(6\), the pointwise inequality
\(\binom d2\ge2d-3\) for \(0\le d\le3\) gives the contradictory lower bounds
\(14g-48=22,36\), whereas the displayed upper bounds are \(20,30\).
The case \(g\ge7\) contradicts \(7g=\sum_xd_x\le16\cdot3\).
For \(g\le3\), summing \eqref{eq:n16-k7-local-en} gives
\[
 c(P)\ge\frac{296}{3}-\frac g6
       =98+\frac{4-g}{6}>98.
\]

It remains only to exclude equality when \(g=4\).  If no point had
\(d_x=1\), then \(d_x\in\{0,2,3\}\) and
\(\binom{d_x}{2}\ge d_x/2\); hence the second intersection moment would be
at least \(14>12\).  Choose a point with \(d_x=1\).  If \(c(P)=98\), equality
must hold in every term used in \eqref{eq:n16-k7-identity-en}.  The positive
residual and the equality cases of Melchior and Bojanowski then give the unique
multiplicity vectors
\[
 (t_2,t_3,t_4,t_5,t_6)=(14,12,0,4,1),\qquad
 (b_2,b_3,b_4,b_5,b_6)=(7,0,0,0,0).
\]
The four five-point lines and one six-point line represented by these vectors have union of
size at least \(4\cdot5+6-\binom52=16\), again impossible on fifteen image
points.  This proves the case \(K=7\).

Now let \(K=8\).  At any point on a chosen largest eight-point line or circle,
inversion sends the other seven points of that block to a single seven-point
image line.  Put
\[
 O=t_2-b_2,\qquad I=\sum_{i=2}^7(i-1)t_i,\qquad
 S=\sum_{i=2}^7b_i,
\]
and define \(T,M\) as above with upper index \(7\).  The ordinary-line bound,
inequality~\eqref{eq:center-en}, and incidence counting with the seven-point
image line give
\[
 O\ge3,\qquad S\le7,\qquad I\ge6+7(15-7)=62.
\]
Another coefficient identity is
\begin{equation}\label{eq:n16-k8-identity-en}
\begin{split}
 \Phi_x={}&\frac1{80}T+\frac{31}{160}M+\frac1{48}O
        -\frac5{16}S+\frac{17}{160}I+R_8,\\
 R_8={}&\frac{t_5}{240}+\frac{3t_6}{560}+\frac{b_3}{16}
 +\frac{9b_4}{80}+\frac{7b_5}{48}+\frac{19b_6}{112}
 +\frac{3b_7}{16}\ge0.
\end{split}
\end{equation}
It follows that each of the eight centres on the chosen block satisfies
\(\Phi_x\ge1017/160\).  At every other centre the unconditional bound
\eqref{eq:bojanowski-local-en} gives \(\Phi_x\ge145/24\).  Therefore
\[
 8\cdot\frac{1017}{160}+8\cdot\frac{145}{24}-98=\frac{71}{60}>0.
\]
Thus \(c(16)\ge99=F(16)\).  The accompanying verifier only expands
\eqref{eq:n16-k7-identity-en} and \eqref{eq:n16-k8-identity-en} and checks the
displayed equality and moment calculations with exact rational arithmetic;
it contains no optimization model or multiplicity-vector enumeration.

\section{The fifteen-point case}\label{sec:n15-en}

We now prove $c(15)=85$.  Let $K$ be the largest number of points of an
admissible set $P$ on one line or circle.  If $K\ge8$,
Lemma~\ref{lem:largest-en} gives $c(P)\ge85$.  We may therefore assume
$K\le7$.

For $x\in P$, invert the other fourteen points about $x$ and dualize the
resulting point set.  This gives an arrangement of fourteen real projective
lines.  Write $t_r(x)$ for its number of $r$-fold vertices, and put
\[
 \delta_x=t_2(x)-3-\sum_{r=4}^{6}(r-3)t_r(x).
\]
Melchior's relation identifies $\delta_x$ with the total excess of the face
degrees above three, so $\delta_x\ge0$.

The fact that the number of lines is even rules out $\delta_x=1$.  Indeed,
the sign of the product of the fourteen homogeneous line equations is
well-defined on the projective plane and gives a two-colouring of the faces.
If $E=\sum_r r t_r(x)$ is the number of edges and $F_+,F_-$ are the two
numbers of faces, then
\[
 \varepsilon_+=E-3F_+,
 \qquad \varepsilon_-=E-3F_-
\]
are nonnegative, congruent modulo $3$, and satisfy
$\delta_x=\varepsilon_++\varepsilon_-$.  Their sum cannot be one.

Define the following integer associated with the centre $x$:
\begin{equation}\label{eq:n15-residual-en}
 R_x=136t_2(x)+93t_3(x)+60t_4(x)+30t_5(x)-2902.
\end{equation}
We first show that $R_x\ge0$.  The point-pair identity and the definition of
$\delta_x$ give
\[
 3t_3+7t_4+12t_5+18t_6=88-\delta_x.
\]
Since the maximum vertex multiplicity is six, the hypothesis of Bojanowski's
inequality is satisfied, and
\[
 4t_2+3t_3-5t_5-12t_6\ge56.
\]
Eliminating $t_2,t_3$ yields
\begin{equation}\label{eq:n15-rich-bound-en}
 3t_4+9t_5+18t_6\le44+3\delta_x
\end{equation}
and
\begin{equation}\label{eq:n15-residual-delta-en}
 R_x=234+105\delta_x-21t_4-70t_5-150t_6.
\end{equation}
If $\delta_x\ge2$, then \eqref{eq:n15-rich-bound-en} gives
\[
21t_4+70t_5+150t_6
 \le25(t_4+3t_5+6t_6)
 \le\frac{25}{3}(44+3\delta_x),
\]
and hence $R_x\ge(240\delta_x-398)/3$.  Since $R_x$ is an integer,
this gives $R_x\ge28$.

It remains to consider $\delta_x=0$, when the arrangement is simplicial.
Cuntz's complete classification through twenty-seven lines
\cite[Theorem~1.1 and the fourteen-line entries of the invariant table,
p.~15]{cuntz-en} shows that, after excluding the type with a sevenfold vertex
and the near-pencil, the possible vectors $(t_2,t_3,t_4,t_5,t_6)$ are
\[
 (11,12,4,2,0),\qquad(9,16,4,1,0),\qquad(10,14,4,0,1).
\]
Their values of $R_x$ are respectively $10,80,0$.  Consequently
\begin{equation}\label{eq:n15-residual-gap-en}
 R_x\ge0,
 \qquad R_x>0\Longrightarrow R_x=10\ \hbox{or}\ R_x\ge28.
\end{equation}

Let $\ell_k$ and $c_k$ denote the numbers of maximal $k$-point lines and
circles, and put $C=\sum_{k=3}^{7}c_k=c(P)$.  We use the exact identities
\[
 \ell_2+\sum_{k=3}^{7}\binom{k}{2}\ell_k=\binom{15}{2},
 \qquad
 \sum_{k=3}^{7}\binom{k}{3}(\ell_k+c_k)=\binom{15}{3},
\]
and
\[
 \sum_{x\in P}t_{k-1}(x)=k(\ell_k+c_k)\qquad(3\le k\le7).
\]
The Csima--Sawyer theorem gives $\ell_2\ge\lceil90/13\rceil=7$
\cite[the theorem on pp.~187--188]{csima-sawyer-en}.  Direct substitution in
the preceding identities gives
\begin{align}\label{eq:n15-global-identity-en}
420C-35270={}&140(\ell_2-7)+420\ell_4+980\ell_5
 +1680\ell_6+2520\ell_7+\sum_{x\in P}R_x.
\end{align}
Every term on the right is nonnegative, so $C\ge84$.

Suppose that $C=84$.  The left side of
\eqref{eq:n15-global-identity-en} is then $10$.  Thus
$\ell_2=7$, $\ell_4=\cdots=\ell_7=0$, and
$\sum_xR_x=10$.  By \eqref{eq:n15-residual-gap-en}, exactly one centre has
local type $(11,12,4,2,0)$, while the other fourteen have type
$(10,14,4,0,1)$.  It follows that
\[
 \sum_{x\in P}t_2(x)=11+14\cdot10=151.
\]
This is impossible because the local incidence identity for $k=3$ says that
the same sum is $3(\ell_3+c_3)$.  Hence $C\ge85$.  The construction in
\eqref{eq:F-en} has $F(15)=85$, and therefore $c(15)=85$.

\begin{remark}
The proof does not choose the point subsets of any lines or circles and does
not enumerate geometric orbits.  The
finite input consists only of the three multiplicity vectors in Cuntz's
complete fourteen-line table; all remaining steps are integer identities and
inequalities.
\end{remark}

\section{The cases \texorpdfstring{$9\le n\le14$}{9<=n<=14}}
\label{sec:small-uniform-en}

This section separates the two larger cases, which are decided at the level of
global integer data, from the four smaller cases that require an exact finite
verification.  No historical search tree or saved orbit list is used.

For $k\ge3$, let $\ell_k$ and $c_k$ denote the numbers of maximal $k$-point
lines and circles, and put $m_k=\ell_k+c_k$.
Every pair belongs to one maximal
line, while every triple belongs either to its maximal line or to its unique
circumcircle.  Therefore
\begin{align}
 \ell_2+\sum_{k\ge3}\binom{k}{2}\ell_k&=\binom n2,
 \label{eq:small-line-pairs-en}\\
 \sum_{k\ge3}\binom{k}{3}m_k&=\binom n3,
 \qquad \sum_{k\ge3}c_k=c(P).
 \label{eq:small-triple-partition-en}
\end{align}
We call $(\ell_3,\ldots,\ell_K;c_3,\ldots,c_K)$ the \emph{global count
vector}; it records only the number of maximal lines and circles of each size,
not the subsets of points on them.  Lemma~\ref{lem:largest-en} excludes maximum blocks of size at
least $6,7,7,7,7,8$ for $n=9,10,11,12,13,14$, respectively.  Thus $K$ is at
most $5,6,6,6,6,7$ in the branches considered below.

The enumeration of these count vectors uses only necessary conditions.  Since $K<n-1$,
every one-point deletion is admissible, and summing its circle count gives
\begin{equation}\label{eq:small-one-delete-en}
 n c(P)-3c_3\ge n c(n-1).
\end{equation}
Inversion followed by Zhang's prescribed-point ordinary-line theorem
\cite[Theorem~4.1]{zhang-en} gives
\begin{equation}\label{eq:small-z-circle-en}
 3c_3\ge n\left\lceil\frac{n-1}{6}\right\rceil.
\end{equation}
We also use the ordinary-line, Melchior, Hirzebruch--Bojanowski, and
Shnurnikov inequalities for the line part of the count vector, and their sums
over subsets that are forced to be noncollinear
\cite[Eq.~(9)]{hirzebruch-en}
\cite[Theorem~2.1]{pokora-en}\cite[Theorem~3]{shnurnikov-en}.
For a subset contained in a maximal line no circle lower bound is imposed; for
a subset contained in a maximal circle only that defining circle is imposed.
This is the correction needed when deletion averages are used at small order.

We shall also use the following elementary consequence of block intersections.
For $N\ge1$ and $S=aN+b$, $0\le b<N$, set
\[
 \Phi_N(S)=N\binom a2+ab.
\]
This is the minimum of $\sum_i\binom{x_i}{2}$ over nonnegative integral
$x_i$ with $\sum_i x_i=S$.  If a selected family consists of $L$ maximal
lines and $G$ maximal circles and $I_1,I_2$ are its total point and point-pair
incidences, then
\begin{align}
 \Phi_n(I_1)&\le \binom L2+2LG+2\binom G2,
 \label{eq:small-block-point-en}\\
 \Phi_{\binom n2}(I_2)&\le LG+\binom G2.
 \label{eq:small-block-pair-en}
\end{align}
Indeed, two lines meet in at most one point and any other two distinct blocks
meet in at most two points.  Applying these inequalities to the largest
members of each subfamily covers every possible choice.  It does not decide
which labelled points lie on the selected lines and circles.

\subsection{The augmented-line inequality}

Fix $x\in P$, invert about $x$, and dualize the remaining $n-1$ points.  Let
$t_j(x)$ be the number of multiplicity-$j$ vertices of the resulting
arrangement and put
\[
 \delta_x=t_2(x)-3-\sum_{j\ge4}(j-3)t_j(x).
\]
Write $d_k^L(x)$ for the number of maximal $k$-point lines of $P$ through
$x$.

\begin{lemma}\label{lem:augmented-line-en}
For every $x\in P$,
\begin{equation}\label{eq:augmented-line-point-en}
 \sum_{k\ge3}k d_k^L(x)\le n-1+\delta_x.
\end{equation}
Consequently,
\begin{equation}\label{eq:augmented-line-summed-en}
 \sum_{k\ge3}k^2\ell_k\le n(n-1)+D,
 \qquad
 D=3m_3-3n-\sum_{k\ge5}k(k-4)m_k.
\end{equation}
\end{lemma}

\begin{proof}
Let $\mathcal A_x$ be the arrangement of $n-1$ dual lines, and adjoin the
projective line dual to the inversion centre.  A $k$-point original line
through $x$ corresponds to an old vertex of multiplicity $k-1$ on the added
line.  Suppose the selected old vertices have multiplicities $i_1,\ldots,i_s$.
The added line has
$(n-1)-\sum_j i_j$ new ordinary vertices.  Raising the multiplicity of each
selected old vertex by one decreases the Melchior defect by one, including
the cases of multiplicity two and three.  Hence the defect of the enlarged
arrangement is
\[
 \delta_x+(n-1)-\sum_{j=1}^s(i_j+1)
 =\delta_x+(n-1)-\sum_{k\ge3}k d_k^L(x).
\]
The enlarged arrangement is essential, so Melchior's inequality proves
\eqref{eq:augmented-line-point-en}.  Summing over $x$, and using
$\sum_xd_k^L(x)=k\ell_k$ and the local incidence identities, gives
\eqref{eq:augmented-line-summed-en}.
\end{proof}

\subsection{Global exclusion for thirteen and fourteen points}

For $n=14$, the exact integer solutions of the preceding necessary
conditions with $c(P)<F(14)$ give 8981 count vectors.  Inequality
\eqref{eq:augmented-line-summed-en} excludes 8980.  The remaining vector is
\[
 (\ell_3,\ldots,\ell_7;c_3,\ldots,c_7)
 =(20,0,0,0,0;23,44,0,2,3).
\]
Its three seven-point circles have pairwise intersections of size at most two,
so their union has at least $3\cdot7-3\cdot2=15$ points, a contradiction.
Thus $c(14)=F(14)$.

For $n=13$, the same exact enumeration gives 3962 count vectors below $F(13)$.
The augmented-line inequality leaves 137, and
\eqref{eq:small-block-point-en}--\eqref{eq:small-block-pair-en} leave 73.  We
now derive two inequalities that exclude these remaining vectors.

At any point $x$, the local arrangement has twelve lines and, in the present
branch, no vertex of multiplicity greater than five.  Write its multiplicity
vector as $(t_2,t_3,t_4,t_5)$ and its defect as $\delta$.  Then
\begin{equation}\label{eq:n13-local-two-ineq-en}
 -2\delta+t_4+3t_5\le6,
 \qquad -\delta+t_5\le1.
\end{equation}
For the first inequality, the pair identity and Bojanowski's inequality give
\begin{align*}
 t_2+3t_3+6t_4+10t_5&=66,\\
 4t_2+3t_3&\ge48+5t_5,
\end{align*}
and hence
\[
 H:=3t_2-6t_4-15t_5+18\ge0.
\]
Put $E=2t_2-3t_4-7t_5$.  Since
$H+3\delta=3(E+3)$, one has $E\ge-3$.  Moreover, the pair identity gives
$E\equiv0\pmod3$.  Equality $E=-3$ would force $H=\delta=0$.
Cuntz's complete list of simplicial twelve-line arrangements, after excluding
a six-fold vertex, leaves the types $(8,10,3,1)$ and $(9,7,6,0)$; both have
$E=0$ \cite[Theorem~1.1 and the invariant tables]{cuntz-en}.  Therefore
$E\ge0$, which is the first inequality in \eqref{eq:n13-local-two-ineq-en}.

For the second inequality, four five-fold vertices would dualize to four
five-point lines whose union has at least $20-\binom42=14$ points; hence
$t_5\le3$.  If $\delta\ge2$ the claim follows.  If $\delta=0$, Cuntz's two
types above have $t_5\le1$.  The only remaining possible violation is
$(\delta,t_5)=(1,3)$.  The pair and defect identities then force
$(t_2,t_3,t_4,t_5)=(12,4,2,3)$.  But three five-point lines and one
four-point line have a union of at least $5+5+5+4-\binom42=13$ points, again
impossible on twelve points.  This proves the second inequality.

Summing \eqref{eq:n13-local-two-ineq-en} over the thirteen centres yields
\begin{equation}\label{eq:n13-global-two-ineq-en}
 -2D+5m_5+18m_6\le78,
 \qquad -D+6m_6\le13.
\end{equation}
Of the 73 count vectors, 71 violate the first inequality and the other two
violate the second.  This is an exact substitution into
\eqref{eq:n13-global-two-ineq-en}; no subsets of labelled points are selected.  Hence
$c(13)=F(13)$.

\subsection{Exact verification for nine through twelve points}

For $9\le n\le12$, the same global constraints and
\eqref{eq:augmented-line-summed-en} leave a small but nonempty boundary.  We
describe the common verification so that every transition can be reproduced.
For a marked point $x$, let $a_k(x)$ and $u_k(x)$ count maximal $k$-point
lines and circles through $x$.  Then
\[
 t_{k-1}(x)=a_k(x)+u_k(x),
 \qquad
 \sum_{j\ge2}\binom j2t_j(x)=\binom{n-1}{2}.
\]
These local multiplicity vectors obey the published arrangement inequalities used above, the
block-union inequalities, and \eqref{eq:augmented-line-point-en}.  A first
integer model records only how many points have each vector and imposes
\[
 \sum_xa_k(x)=k\ell_k,\qquad
 \sum_xu_k(x)=kc_k,\qquad
 \sum_x\delta_x=D.
\]

The next two stages still do not assign labelled point subsets.  For a block category
$\alpha$, let $d_\alpha(x)$ be its point degree and let
$\lambda_\alpha(e)$ be its degree at a point pair $e$.  If $k_\alpha$ is the
block size and $N_\alpha$ its multiplicity, then
\begin{align}
 \sum_xd_\alpha(x)&=k_\alpha N_\alpha,&
 \sum_e\lambda_\alpha(e)&=\binom{k_\alpha}{2}N_\alpha,
 \label{eq:small-category-totals-en}\\
 \sum_\alpha(k_\alpha-2)\lambda_\alpha(e)&=n-2.
 \label{eq:small-pair-profile2-en}
\end{align}
The last identity says that the portions outside $e$ of all maximal blocks
through $e$ partition $P\setminus e$.  The point and point-pair data are linked by
\[
 \sum_{y\ne x}\lambda_\alpha(\{x,y\})=(k_\alpha-1)d_\alpha(x).
\]
The subsequent intersection count uses,
for two block categories, the exact identities
\[
 N_2=M_2,\qquad N_1=M_1-2M_2,\qquad N_0=Q-M_1+M_2,
\]
where $M_1$ is the sum of intersection sizes, $M_2$ is the sum of their second
binomial coefficients, and $N_s$ counts pairs of
distinct blocks meeting in $s$ points.  It also counts triples disjoint from
a distinguished block.  All coefficients are binomial counts, and no
floating-point relaxation is used.

Finally, inversion at $x$ determines the complete count vector on $n-1$ image
points: blocks through $x$ become image lines with one fewer point, whereas
blocks avoiding $x$ become image circles.  A local vector is retained only if this
child vector survives the next level, after which the local multiplicity data
at the $n$ points are balanced again.  The recursion stops at $n=9$.  There the model
selects maximal line and circle subsets of the nine labelled points, requires
every triple to lie in exactly one selected block, requires a point pair to
lie in at most one selected line, and imposes the already proved lower bounds
on every admissible proper subset.  Infeasibility of this more permissive
abstract model implies geometric infeasibility.

The exact counts in the final replay are:
\begin{center}
\small
\begin{tabular}{c|rrrrrrr}
\toprule
$n$ & initial & augmented sum & local vectors & point--pair & intersections & block pairs & terminal\\
\midrule
9  &45   &36  &23  &-- &-- &-- &0\\
10 &180  &150 &55  &26 &21 &12 &1\\
11 &416  &160 &70  &39 &28 &-- &0\\
12 &1354 &211 &169 &55 &44 &-- &0\\
\bottomrule
\end{tabular}
\end{center}
Here ``initial'' is the number left from 129, 551, 1215, and 4322 raw
count vectors, respectively.  The terminal zeroes for $n=11,12$ come from the
same recursion.  It visits, respectively,
\[
 (28,72,222)\quad\text{and}\quad(44,322,661,622)
\]
count vectors at orders $(11,10,9)$ and $(12,11,10,9)$.  Every excluded status
is an exact \texttt{INFEASIBLE} status; there is no \texttt{UNKNOWN}.

For $n=10$ the recursion leaves the single count vector
\[
 (\ell_3,\ell_4,\ell_5,\ell_6;
   c_3,c_4,c_5,c_6;\ell_2)
 =(10,0,0,0;10,20,2,0;15).
\]
The four local multiplicity vectors retained by the recursive certificate all have
$u_5(x)=1$ and $a_3(x)\le3$.  Since the two five-point circles have ten point
incidences, they therefore partition the ten points into sets $A$ and $B$.
Moreover, the ten three-point lines have thirty point incidences, so
$a_3(x)=3$ at every point; the corresponding retained vector also has
$u_4(x)=8$.  Every four-point circle meets each of $A$ and $B$ in two points.

Associate such a circle with its chord in $A$ and its chord in $B$.  The two
chords meet on the fixed radical axis of the two five-point circles.  No three
chords determined by five points on a circle can pass through the same point
off that circle: three pairwise disjoint chords would require six endpoints.
Thus each set of chords meeting at one radical-axis point has size at most two
on either side.  The twenty four-point circles attain the resulting maximum,
so the chord graph is five disjoint copies of $K_{2,2}$.  Up to the action of
$S_5\times S_5$, the six perfect matchings of the Petersen graph give two
double orbits.  One admits no compatible ten-line design; the other has two
complementary designs forming one orbit.

For a representative of that orbit, the ten line incidences give a complete
projective parametrization with nonzero parameters $a,b,d$.  Five selected
circle-intersection minors are
\begin{gather*}
 ad+b^2,\quad -ab+ad-b,\quad ab-ad-a-1,\\
 -ad+bd-b-d,\quad b(ab+1).
\end{gather*}
They must all vanish.  Adding the second and third relations gives
$a+b+1=0$, while the last relation and $b\ne0$ give $ab+1=0$.  Hence
$b^2+b-1=0$.  The first relation gives
$d=b^2/(b+1)$; here $b+1=-a\ne0$.  Substitution in the fourth relation then
gives, after removal of nonzero factors,
\[
 (b-1)(2b+1)=0,
\]
which is relatively prime to $b^2+b-1$.  The discarded parameter branch makes
two labelled points coincide, and every factor divided out in forming the
five primitive minors is one of $a,d,b$ or $ad$, all nonzero.  Thus the last
$n=10$ count vector is impossible.

Combining these verifications with the construction in \eqref{eq:F-en} gives
\[
 c(n)=F(n)\qquad(9\le n\le14).
\]

\begin{remark}[Scope of the finite verification]
The only labelled search in this section occurs in the nine-point terminal
model.  Orders ten through twelve use count vectors and recursive consistency
of local multiplicities, while orders thirteen and fourteen use only global inequalities.
No condition $\sum_jt_j\ge2r-4$ is assumed for an arbitrary arrangement.
For an arrangement of $q$ lines, Elliott's Theorem~3 is used only when its
hypotheses $q\ge10$ and no $(q-1)$-fold vertex are explicitly satisfied.
\end{remark}

\section{The finite cases \texorpdfstring{$4\le n\le8$}{4<=n<=8}}

We shall use the following two elementary facts repeatedly.
\begin{lemma}\label{lem:two-line-geometry-en}
Let $A,B,C,D$ be four points in the real projective plane, no three collinear.
Then the three points $AB\cap CD$, $AC\cap BD$, and $AD\cap BC$ are not
collinear.  Moreover, suppose a circle meets each of two lines in two labelled
points.  If the lines intersect and directed unit coordinates are measured from
their intersection, the products of the two coordinate pairs are equal.  If the
lines are parallel and a common directed coordinate is used, the sums of the two
coordinate pairs are equal.  Finally, if three Euclidean circles have three
distinct pairwise radical axes, those axes are concurrent in the real projective
completion of the plane.
\end{lemma}
\begin{proof}
A projective change of coordinates sends $A,B,C,D$ to
$(1,0,0),(0,1,0),(0,0,1),(1,1,1)$.  The three displayed intersections then have
coordinates $(1,1,0),(1,0,1),(0,1,1)$; their determinant is $-2$.
For two intersecting lines, use them as oblique coordinate axes.  A circle has
equation
\[
x^2+y^2+2kxy+Dx+Ey+G=0,
\]
so its restrictions to the two axes are monic quadratics with the same constant
term $G$.  For parallel lines $y=0$ and $y=h$, the two restrictions are monic
quadratics in $x$ with the same linear coefficient $D$.  Vieta's formulas give
the two assertions.  For the last assertion, normalize the equations of the
three circles to have quadratic part $x^2+y^2$.  The three pairwise differences
are linear equations whose sum is zero.  Hence the intersection of any two of
the corresponding projective lines lies on the third.
\end{proof}

For $n=4$, either three points are collinear, in which case the remaining point
with each of the three pairs on that line gives three distinct circles, or no three
points are collinear, in which case the four triples give either one four-point
circle or four three-point circles.  Admissibility excludes the former alternative,
so $c(4)\ge3$.  The construction in \eqref{eq:F-en} attains three.  For $n=5$,
four collinear points give six circles through the fifth point, while four concyclic
points give at least $1+\binom42-\lfloor4/2\rfloor=5$ circles by
Lemma~\ref{lem:largest-en}.  In the remaining case every line and circle has at most
three points.  There are at most two collinear triples: three distinct three-point
lines on five points would force two of them to share two points.  Hence at least
eight noncollinear triples remain; because no four points are concyclic, they
determine distinct circles.  Thus $c(5)\ge5$, and \eqref{eq:F-en} attains five.

For six points, the direct estimate of Lemma~\ref{lem:largest-en} gives at least
nine circles whenever four points are collinear or five points are collinear or
concyclic.  If no four points are concyclic and every line contains at most three
points, at most four collinear triples can occur: each point is on at most two
three-point lines, so the total number of point--line incidences is at most twelve.
Thus at least $20-4=16$ noncollinear triples determine distinct circles.  It remains
to consider four points $a_0,a_1,a_2,a_3$ on a circle $\Gamma$, with the other
points denoted $q,r$, under the assumptions that no four are collinear and no five
are concyclic.  Let $\tau_q,\tau_r$ count the secants through $q,r$, let $s$ record
whether $qr$ contains an $a_i$, and let $b$ count circles through $q,r$ and two
$a_i$'s.  Disjointness of the relevant pairs gives
$\tau_q,\tau_r,b\le2$ and $s\le1$, while triple counting gives
\[
c(P)=17-\tau_q-\tau_r-s-3b.
\]
If $c(P)\le7$, then $b=2$.  Relabel the two disjoint pairs as $01$ and $23$,
and let $C_{01}$ and $C_{23}$ be the circles through
$q,r,a_0,a_1$ and $q,r,a_2,a_3$, respectively.  The three pairwise radical
axes of $\Gamma,C_{01},C_{23}$ are $a_0a_1,a_2a_3,qr$; hence these three
lines are concurrent by Lemma~\ref{lem:two-line-geometry-en}, in the projective
completion.  If $s=1$, this concurrency makes $qr$ equal to one of the first two
lines, producing four collinear points.  Thus $s=0$, and the displayed count with
$c(P)\le7$ forces $\tau_q=\tau_r=2$.  The points $q,r$ must be two diagonal points of the complete
quadrilateral on $a_0,a_1,a_2,a_3$, whereas the same radical-axis relation would
make all three diagonal points collinear, contrary to
Lemma~\ref{lem:two-line-geometry-en}.  Hence $c(6)\ge8$,
and the construction below attains eight.

For seven points, the largest-block estimate first removes every case with at least
five points on a line or circle from a putative configuration with at most ten
circles.  When all blocks have size at most four, write $\ell_3,\ell_4$ for the
numbers of maximal three- and four-point lines, put
$\ell=\ell_3+4\ell_4$, and let $c_4$ be the number of four-point circles.  Unique
ownership of triples gives
\[
c(P)=35-\ell-3c_4.
\]
There are at most two four-point lines, and the pair identity gives
$3\ell_3+6\ell_4\le21$; hence $\ell\le11$.  Thus $c(P)\le10$ requires at least
five four-point circles.  The exact isomorphism classification of such circle
families gives two classes with five circles and one class each with six and seven
circles.  Neither five-circle class has enough compatible collinear triples.  The
six-circle class has a unique extension with $c(P)\le10$; its seven three-point
lines cover all $21$ point pairs, contrary to the Sylvester--Gallai theorem
\cite[pp.~111--112]{borwein-moser-en}.

It remains to exclude the seven-circle family.  Every compatible extension with
$c(P)\le10$ contains, after relabeling, the lines $016$ and $025$.  Normalize
$p_0=(0,0)$, $p_1=(1,0)$, $p_2=(0,1)$, $p_6=(a,0)$ and $p_5=(0,b)$.  Since this
is an affine normalization, write the common quadratic part of every circle as
\[
 Q(x,y)=x^2+2kxy+hy^2,\qquad h-k^2>0.
\]
The seven prescribed circle equations first give $a=bh$ and, with
$t=(h-a)/(a-1)$,
\[
 x_3=ty_3,\qquad x_4=-ty_4,
\]
and the first two equations then give
\[
 y_3=\frac{t+h}{Q(t,1)},\qquad
 y_4=\frac{a-t}{Q(-t,1)}.
\]
Positive definiteness makes both denominators positive.  The last two circle
equations reduce to
\[
ah-ak-hk+h=0,\qquad ak-a-h+k=0.
\]
Distinctness gives $a\ne0,1$, $b\ne1$, and $y_3\ne0$.  Hence the numerator
$t+h$ is nonzero; since $t+h=a(h-1)/(a-1)$, one has $h\ne1$.
Adding the two displayed equations now
gives $a=k$, and the second gives $h=k^2$, contradicting positive definiteness.
The exact reductions are checked in Appendix~\ref{app:verification-en}.  Hence
$c(7)\ge11$; the construction below attains eleven.

For eight points, Lemma~\ref{lem:largest-en} removes every line or circle of size
at least five from any putative example with at most sixteen circles.  Thus all
remaining blocks have size three or four.  Let $\ell$ be the number of collinear
triples and let $c_4$ count four-point circles.  Then
\[
c(P)=56-\ell-3c_4.
\]
In fact the range below twelve can be excluded without a finite search.  Write
$m=\ell_4+c_4$ for the number of four-point blocks and let $\ell_3$ count
three-point lines.  The four triples owned by distinct four-point blocks are
disjoint, so $m\le14$, and at most $56-4m$ triples remain available for
three-point lines.  Hence
\[
 \ell_3\le\min\{7,56-4m\},
\]
where the first bound is the eight-point orchard bound
\cite[Table~I]{burr-grunbaum-sloane-en}.  Moreover, three four-point lines
would have union of size at least $12-3=9$, since two distinct lines meet in
at most one point; thus $\ell_4\le2$.  Finally,
$\ell_2+3\ell_3+6\ell_4=28$ and Melchior's inequality
$\ell_2\ge3+\ell_4$ show that $\ell_4=2$ implies $\ell_3\le3$.
Since $\ell=\ell_3+4\ell_4$ and $c_4=m-\ell_4$, the preceding circle-count
identity becomes
\[
 c(P)=56-3m-\ell_3-\ell_4.
\]

If $c(P)\le11$, then
\[
 3m+\ell_3+\ell_4\ge45.
\]
For $m\le11$ the left side is at most $33+7+2=42$.  For $m=12$ it is
at most $44$ when $\ell_4\le1$, and at most $41$ when $\ell_4=2$.
For $m=13$, the bound $\ell_3\le4$, together with the preceding improvement
when $\ell_4=2$, gives at most $44$.  For $m=14$ one has $\ell_3=0$, again
giving at most $44$.  This proves $c(P)\ge12$ mathematically.

One uniform finite model now classifies only the five boundary counts from
twelve through sixteen.  It has two exhaustive branches: a four-point line is
fixed up to relabeling, or no four-point line exists.  Counts twelve and thirteen
each have one incidence class; counts fourteen, fifteen, and sixteen have,
respectively, one, three, and three classes.  For the
twelve-circle class, its two
four-point lines and the listed circles give the product or sum relations associated
with the three pairings of four distinct coordinates; any two already force a
repeated coordinate by Lemma~\ref{lem:two-line-geometry-en}.  In the
thirteen-circle class, paired radical axes put all three diagonal points of a
complete quadrilateral on the four-point line, contradicting the first assertion of
that lemma.  For the fourteen-circle class, choose an inversion centre on one
selected circle but away from the eight points and its finitely many intersections
with the other circles.  Exactly that circle becomes a four-point line, and the
resulting incidence class is the excluded thirteen-circle class.  At fifteen
circles, one class is excluded by the exact inversion equations
detailed in Appendix~\ref{app:verification-en}; the other two reduce to the same
product/sum relations.  At sixteen circles, the two classes containing a four-point
line are excluded by the complete-quadrilateral argument.  For the remaining class,
let $E_{ijkl}$ be the numerator of the Euclidean concyclicity determinant for
$p_i,p_j,p_k,p_l$ after the stated coordinate substitution.  The exact formulas
factor $E_{ijkl}$ into factors already known to be nonzero and one residual
polynomial, denoted by $F_{ijkl}$.  Thus the required circle is equivalent to
$F_{ijkl}=0$ in this branch.  These residual polynomials satisfy
\[
F_{0256}-F_{2357}=2a(u-b),\qquad
F_{0256}+F_{0367}\bigm|_{u=b}=2s,
\]
where the coordinate normalization has $a\ne0$ and $s\ne0$; hence the three
required circle equations are inconsistent.  These branches exhaust all incidence
classes through sixteen circles, so $c(8)\ge17$.  The rational construction below
attains seventeen.

\begin{remark}[What the finite search proves]
The integer models classify abstract maximal line and circle blocks subject to
unique triple ownership and the intersection bounds for two lines, a line and a
circle, or two circles.  Feasibility of such a model is not treated as geometric
realizability.  Every surviving class is therefore followed by a separate exact
geometric exclusion.  Conversely, an infeasibility conclusion is used only when the
solver reports an exhaustive status; a time limit or an unknown status is rejected.
\end{remark}

\subsection{\texorpdfstring{$c(6)=8<9=F(6)$}{c(6)=8<9=F(6)}}

Let $s=\sqrt3$ and take
\[
\begin{aligned}
p_0&=(7,0),&p_1&=(1,0),&p_2&=(4,2s),\\
p_3&=(1,s/2),&p_4&=(1,4s),&p_5&=(1/7,4s/7).
\end{aligned}
\]
\begin{figure}[ht]
\centering
\begin{tikzpicture}[x=0.72cm,y=0.72cm,every node/.style={font=\small}]
  \coordinate (p0) at (7,0);
  \coordinate (p1) at (1,0);
  \coordinate (p2) at (4,3.4641);
  \coordinate (p3) at (1,0.8660);
  \coordinate (p4) at (1,6.9282);
  \coordinate (p5) at (0.1429,0.9897);
  \draw[gray] (p0)--(p4);
  \draw[gray] (p0)--(p5);
  \draw[gray] (p1)--(p4);
  \fill (p0) circle (2.1pt) node[below right] {$p_0$};
  \fill (p1) circle (2.1pt) node[below] {$p_1$};
  \fill (p2) circle (2.1pt) node[right] {$p_2$};
  \fill (p3) circle (2.1pt) node[right] {$p_3$};
  \fill (p4) circle (2.1pt) node[above] {$p_4$};
  \fill (p5) circle (2.1pt) node[left] {$p_5$};
\end{tikzpicture}
\caption{The six-point exceptional construction.  Gray segments show its
collinear triples; circles are omitted to keep the point diagram legible.}
\label{fig:c6-en}
\end{figure}
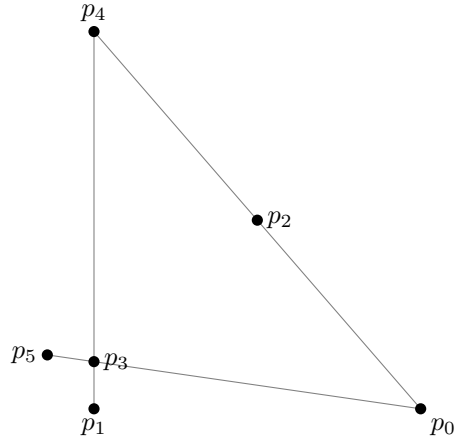
The circles containing at least three of the points are exactly
\begin{equation}\label{eq:c6-en}
0123,\ 0145,\ 025,\ 034,\ 124,\ 125,\ 135,\ 2345.
\end{equation}
There are $\binom63-3=17$ noncollinear triples, and the three four-point circles and
five three-point circles in \eqref{eq:c6-en} contain
$3\binom43+5=17$ triples.  Thus the set determines exactly eight circles.
The program \file{computation/n4_to_8/n6/verify_construction.py} verifies all collinearity determinants,
circle equations, and triple covers exactly in $\mathbb Q(\sqrt3)$; the preceding
case analysis gives the matching lower bound $c(6)\ge8$.

\subsection{\texorpdfstring{$c(7)=11<13=F(7)$}{c(7)=11<13=F(7)}}

Let $s=\sqrt3$ and take
\[
\begin{aligned}
p_0&=(0,0),&p_1&=(1,0),&p_2&=(1/2,-s/2),\\
p_3&=(1/2,s/2),&p_4&=(1/2,-s/6),&
p_5&=(3/2,-s/2),&p_6&=(-1/2,-s/2).
\end{aligned}
\]
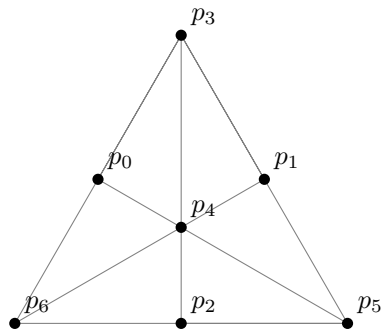
\begin{figure}[ht]
\centering
\begin{tikzpicture}[x=2.2cm,y=2.2cm,every node/.style={font=\small}]
  \coordinate (p0) at (0,0);
  \coordinate (p1) at (1,0);
  \coordinate (p2) at (0.5,-0.8660);
  \coordinate (p3) at (0.5,0.8660);
  \coordinate (p4) at (0.5,-0.2887);
  \coordinate (p5) at (1.5,-0.8660);
  \coordinate (p6) at (-0.5,-0.8660);
  \draw[gray] (p0)--(p3)--(p6);
  \draw[gray] (p0)--(p4)--(p5);
  \draw[gray] (p1)--(p3)--(p5);
  \draw[gray] (p1)--(p4)--(p6);
  \draw[gray] (p2)--(p3);
  \draw[gray] (p2)--(p5)--(p6);
  \foreach \i in {0,...,6}{\fill (p\i) circle (2.1pt) node[above right] {$p_\i$};}
\end{tikzpicture}
\caption{The seven-point construction with eleven circles.  Gray segments show
all collinear triples; circles are omitted.}
\label{fig:c7-en}
\end{figure}
The circles are exactly
\[
0134,\ 0156,\ 0235,\ 0246,\ 1236,\ 1245,
012,\ 345,\ 346,\ 356,\ 456.
\]
The first six contain four points, and the last five are ordinary circles.
Thus the configuration determines exactly eleven circles.  The program
\file{computation/n4_to_8/n7/verify_extremal_configuration.py}
checks these incidences twice: first over the positive-definite form
$x^2+xy+y^2$, and then after the explicit Euclidean change of coordinates
$(x,y)\mapsto(x+y/2,\sqrt3\,y/2)$.

It is natural to ask for an intermediate construction with
$c(P)=12=F(7)-1$.  No such configuration exists.  If five or more points lie
on one line or circle, Lemma~\ref{lem:largest-en} gives at least thirteen
circles.  Otherwise, if $c_4$ is the number of four-point circles and
$\ell=\ell_3+4\ell_4$, then $c(P)=35-\ell-3c_4$.  The pair bound gives
$\ell\le11$.  Equality with twelve circles leaves the cases
$(c_4,\ell)=(4,11),(5,8),(6,5),(7,2)$.  The first would leave no ordinary line,
contrary to the Sylvester--Gallai theorem
\cite[pp.~111--112]{borwein-moser-en}.  The two five-circle packings admit no compatible
line family.  The five six-circle incidence orbits are excluded by the
two-line relations of Lemma~\ref{lem:two-line-geometry-en} and four exact
coordinate reductions; in one reduction the nominally missing sixth line is
forced, producing the eleven-circle configuration above.  The unique
seven-circle orbit is excluded by the positive-definite calculation already
used in the lower bound.  These finite assertions are reproduced by
\file{computation/n4_to_8/n7/classify_line_extensions.py},
\file{verify_circle_count_spectrum.py}, and \file{verify_q7_exclusion.py}.

\subsection{\texorpdfstring{$c(8)=17<19=F(8)$}{c(8)=17<19=F(8)}}

We first make the historical Segre example explicit.  Take the vertices of two
concentric, similarly oriented squares,
\[
(\pm1,\pm1),\qquad(\pm r,\pm r),\qquad r>0,\quad r\ne1.
\]
This is the incidence structure obtained by a central projection of a cube along an
axis through the centres of two opposite faces, as recorded by Elliott
\cite[p.~182, discussion following Theorem~2]{elliott-en}.  It has two four-point
lines, ten four-point circles, and eight three-point circles.  The number of
noncollinear triples is
\[
\binom83-2\binom43=48=10\binom43+8,
\]
so it determines 18 circles.

The following rational construction improves this to 17.  Take
\[
\begin{aligned}
p_0&=(0,0),\\
p_1&=\left(\frac{263}{626},\frac{2178}{4069}\right),&
p_2&=\left(\frac{263}{313},\frac{4356}{4069}\right),\\
p_3&=\left(\frac{789}{626},\frac{6534}{4069}\right),&
p_4&=\left(\frac{53519}{195938},\frac{1842342}{1273597}\right),\\
p_5&=\left(\frac{184032}{458545},\frac{7245468}{5961085}\right),&
p_6&=\left(\frac{160557}{917090},\frac{5527026}{5961085}\right),\\
p_7&=\left(-\frac{25}{313},\frac{312}{313}\right).
\end{aligned}
\]
\begin{figure}[ht]
\centering
\begin{tikzpicture}[x=3.0cm,y=2.0cm,every node/.style={font=\small}]
  \coordinate (p0) at (0,0);
  \coordinate (p1) at (0.4201,0.5353);
  \coordinate (p2) at (0.8403,1.0705);
  \coordinate (p3) at (1.2604,1.6058);
  \coordinate (p4) at (0.2731,1.4465);
  \coordinate (p5) at (0.4013,1.2155);
  \coordinate (p6) at (0.1751,0.9272);
  \coordinate (p7) at (-0.0799,0.9968);
  \draw[gray] (p0)--(p3);
  \draw[gray] (p0)--(p4);
  \draw[gray] (p3)--(p7);
  \fill (p0) circle (2.1pt) node[below] {$p_0$};
  \fill (p1) circle (2.1pt) node[below right] {$p_1$};
  \fill (p2) circle (2.1pt) node[below right] {$p_2$};
  \fill (p3) circle (2.1pt) node[right] {$p_3$};
  \fill (p4) circle (2.1pt) node[above] {$p_4$};
  \fill (p5) circle (2.1pt) node[above right] {$p_5$};
  \fill (p6) circle (2.1pt) node[left] {$p_6$};
  \fill (p7) circle (2.1pt) node[left] {$p_7$};
\end{tikzpicture}
\caption{The rational eight-point construction with 17 circles.  Gray segments
show all maximal collinear subsets; the circles are listed in the text.}
\label{fig:c8-en}
\end{figure}
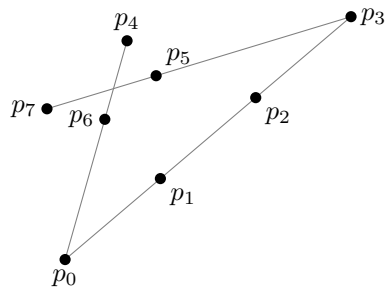
Its 17 circles are
\begin{gather*}
0145,0167,024,0257,026,0347,0356,1247,1256,\\
1346,135,137,157,2345,2367,246,4567.
\end{gather*}
The program \file{computation/n4_to_8/n8/verify_construction.py} uses rational determinants to verify
these circles and checks that all $\binom83=56$ triples belong to exactly one maximal
collinear subset or one listed circle.  The isomorphism classification of all
eight-point incidence structures with at most 16 circles finds none that is
realizable in the Euclidean plane.  Therefore $c(8)\ge17$, and hence $c(8)=17$.

\section{The variant with no three collinear points}
\label{sec:no-three-en}

We now prove Corollary~\ref{thm:no-three-en}.

\begin{proof}
For the general upper bound, take $n-1$ points on a circle and an exterior
point that avoids the finitely many secants through pairs of those points.
There are no three collinear points, and the construction determines the
large circle and one distinct circle for every pair on it.  Thus
$c_{\mathrm{nc}}(n)\le1+\binom{n-1}{2}$.

We first prove the matching lower bound for $n\ge11$.  Here $K$ is the maximum
number of concyclic points, because every line contains at most two points.
If $3(K-1)<2(n-1)$, the proof of
Lemma~\ref{lem:bojanowski-local-en} applies with $b_i(x)=0$ and gives
\[
 \Phi_x\ge
 \frac{K+1}{18K}\binom{n-1}{2}
 +\frac{K+1}{2K}+\frac{(K-2)(n-1)}{9K}.
\]
After summing over $x$ and subtracting $1+\binom{n-1}{2}$, the margin is
\[
 \frac{(n-4)\{Kn^2-13Kn+18K+n^2-7n\}}{36K}.
\]
It decreases with $K$.  At $K=(2n+1)/3$ it equals
\[
 \frac{(n-4)(n-1)(n^2-10n-9)}{18(2n+1)},
\]
which is positive for $n\ge11$.

Suppose instead that $3(K-1)\ge2(n-1)$.  An exterior point is not collinear
with a pair on a largest circle.  The proof of Lemma~\ref{lem:largest-en}
therefore improves to
\[
 c(P)\ge1+(n-K)\binom K2-\binom{n-K}{2}\left\lfloor\frac K2\right\rfloor
 \ge1+\frac{(n-K)K(3K-n-1)}4.
\]
On $(2n+1)/3\le K\le n-1$, the derivative of the last expression is a
concave quadratic, positive at the left endpoint and negative at the right
endpoint.  Hence its minimum is attained at an endpoint.  Its margins above
$1+\binom{n-1}{2}$ there are, respectively,
\[
 \frac{(n-4)(n-1)(2n-9)}{36}\quad\hbox{and}\quad0.
\]
This proves the result for $n\ge11$.

  The remaining orders require only short arguments.  For $n=4$, the four
  triples determine four distinct circles, since putting two triples on one
  circle would make all four points concyclic.  For $n=5$, at most one
  four-point circle can occur, so partitioning the ten triples among the
  circles gives at least $10-3=7$ circles.  For $n=6$, if there is a
  five-point circle, the improved largest-circle bound gives eleven.  Otherwise
  every nonordinary circle has four points.  Two such circles share at most two
  points, so their complementary pairs in the six-point set are disjoint;
  hence there are at most three four-point circles.  Triple counting again
  gives $c(P)\ge20-3\cdot3=11$.  For $n=7$, the only potentially
smaller case would consist of seven four-point circles.  Up to relabelling
these are the complements of the lines of the Fano plane.  Inverting at one
point gives the four lines $123,145,246,356$ and the three circles
$1256,1346,2345$.  Normalize
\[
 p_1=(0,0),\quad p_2=(1,0),\quad p_3=(t,0),\quad
 p_4=(u,v),\quad p_5=a(u,v),
\]
with $p_6=p_2p_4\cap p_3p_5$.  After nonzero incidence factors are removed,
the three circle determinants are
\[
 -a(u^2+v^2)+2tu-t,\qquad
 a(u^2+v^2)-2au+t,\qquad
 -a(u^2+v^2)+t.
\]
They imply successively $u=1$, $v=0$, and $a=t$, contradicting
noncollinearity.  The finite classification up to relabelling is reproduced by
\file{computation/no_three_collinear/verify_no_three_collinear.py}; hence
$c_{\mathrm{nc}}(7)=16$.

For $n=9$, if every circle has at most four points and $c_4$ circles have four
points, then $c(P)=\binom93-3c_4$.  To bound $c_4$, fix a point.  The
four-point circles through it induce three-subsets of the other eight points
with no repeated pair, so their number is at most
$\lfloor(8/3)\lfloor7/2\rfloor\rfloor=8$.  Summing this bound over the nine
points and dividing by four gives
\[
 c_4\le\left\lfloor\frac94
 \left\lfloor\frac83\left\lfloor\frac72\right\rfloor\right\rfloor
 \right\rfloor=18,
\]
so $c(P)\ge30$.  A circle with at least five points is covered by the
improved largest-circle bound, whose minimum is the required $29$.

For $n=10$, the same largest-circle bound deals with $K\ge6$.  If $K\le5$,
invert at $x$ and write $(t_2,t_3,t_4)$ for the multiplicities of connecting
lines among the nine image points.  Pair counting, Melchior's inequality, and
Hirzebruch's inequality give
\[
 t_2+3t_3+6t_4=36,\qquad t_2\ge3+t_4,\qquad t_2\ge2t_4.
\]
Since $t_2\equiv0\pmod3$, direct substitution for $0\le t_4\le4$ yields
\[
 \frac{t_2}{3}+\frac{t_3}{4}+\frac{t_4}{5}\ge\frac{18}{5},
\]
with equality only at $(t_2,t_3,t_4)=(6,4,3)$.  If $c(P)\le36$, equality
must hold at all ten inversion centres.  Consequently there would be six
five-point circles, each point would lie on three of them, and every pair of
these circles would meet in exactly two points.  The resulting
$2$-$(6,3,2)$ incidence design has one relabelling orbit, as verified by the
exact finite enumeration in
\file{computation/no_three_collinear/verify_no_three_collinear.py}.

Invert at a point incident with circles $0,1,4$.  The remaining nine points
lie four at a time on the three sides of a triangle.  After an affine
normalization, write the two additional points on the three sides as
\[
 (u,0),(v,0),\qquad(0,s),(0,t),\qquad(p,1-p),(r,1-r),
\]
and retain the general positive-definite quadratic part
$x^2+2kxy+hy^2$.  The first two five-point-circle equations give
$hs=uv$ and $hst=u$, hence $t=1/v$ and $h=uv/s$.  Eliminating $k$ from the
remaining equations gives
\[
 p(s+1)=1,\qquad r(u+1)=u,\qquad (v-1)(su-1)=0.
\]
Distinctness gives $v\ne1$, so $su=1$ and therefore $p=r$, again a repeated
point.  Thus $c_{\mathrm{nc}}(10)=37$.

It remains to determine the exceptional order eight.  Begin with the two
concentric squares $(\pm1,\pm1),(\pm2,\pm2)$, which determine eighteen
circles and two maximal lines.  Invert about $z=(0,10)$.  The point $z$ lies
on none of these twenty generalized circles, so all twenty become circles and
the image has no collinear triple.  Explicitly, the image points are
\[
\begin{gathered}
(1/82,811/82),(1/122,1209/122),(-1/82,811/82),(-1/122,1209/122),\\
(1/34,168/17),(1/74,367/37),(-1/34,168/17),(-1/74,367/37).
\end{gathered}
\]
Figure~\ref{fig:no-three-eight-en} shows a translated and uniformly rescaled
copy of these exact coordinates.  With the index-string convention of
Section~2, its twenty circles are
\begin{gather*}
0123,0145,0167,0246,0257,0347,035,036,056,124,\\
1256,127,1346,1357,147,2345,2367,247,356,4567.
\end{gather*}
\begin{figure}[ht]
\centering
\begin{tikzpicture}[scale=0.82,every node/.style={font=\small}]
  \coordinate (p0) at (1.2195,1.0244);
  \coordinate (p1) at (0.8197,2.9836);
  \coordinate (p2) at (-1.2195,1.0244);
  \coordinate (p3) at (-0.8197,2.9836);
  \coordinate (p4) at (2.9412,0.2353);
  \coordinate (p5) at (1.3514,3.8919);
  \coordinate (p6) at (-2.9412,0.2353);
  \coordinate (p7) at (-1.3514,3.8919);
  \fill (p0) circle (2.2pt) node[right] {$p_0$};
  \fill (p1) circle (2.2pt) node[right] {$p_1$};
  \fill (p2) circle (2.2pt) node[left] {$p_2$};
  \fill (p3) circle (2.2pt) node[left] {$p_3$};
  \fill (p4) circle (2.2pt) node[right] {$p_4$};
  \fill (p5) circle (2.2pt) node[above right] {$p_5$};
  \fill (p6) circle (2.2pt) node[left] {$p_6$};
  \fill (p7) circle (2.2pt) node[above left] {$p_7$};
\end{tikzpicture}
\caption{The eight-point construction for the no-three-collinear variant.
Only the points are drawn; displaying all twenty circles would obscure the
configuration.}
\label{fig:no-three-eight-en}
\end{figure}
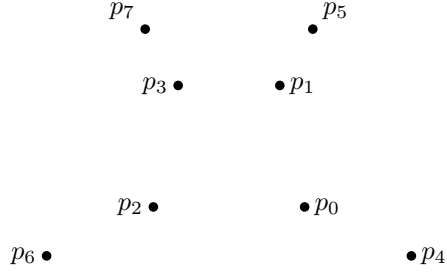
Hence $c_{\mathrm{nc}}(8)\le20$.

For the lower bound, circles of size at least five are again excluded by the
largest-circle estimate.  If all circles have size at most four, then
$c(P)=56-3c_4$.  Theorem~\ref{thm:main-en} gives $c(P)\ge17$, so a value below
twenty would have to be seventeen.  The exact four-subset classification has
one line-free class at that value.  After inversion and relabelling it has the
six lines
\[
036,045,135,146,234,256
\]
and the seven required circles
\[
0134,0156,0235,0246,1236,1245,3456.
\]
Normalize $p_0=(0,0)$, $p_1=(1,0)$, $p_3=(0,1)$, and write
$p_6=(0,a)$ and $p_5=(b,1-b)$.  The line incidences determine the other
points.  For the general quadratic part $x^2+2kxy+hy^2$, the first two circle
conditions give
\[
 h=-\frac{b}{a-b+1},\qquad k=\frac{1-b}{a-b+1}.
\]
After these substitutions, the numerator of the $3456$ circle determinant is
\[
 2ab^2(a-1)^2(b-1)(ab-b+1),
\]
whose factors are all nonzero by point distinctness and the stated line
incidences.  The seventeen-circle class is therefore not realizable without a
collinear triple.  This proves $c_{\mathrm{nc}}(8)=20$ and completes the proof.
\end{proof}

\begin{remark}
The packing classifications and every polynomial identity in the preceding
proof, as well as the rational eight-point construction, are reproduced by
\file{computation/no_three_collinear/verify_no_three_collinear.py}.  The script uses
exact integer and symbolic arithmetic; no numerical root decision is used.
\end{remark}

\appendix
\section{Exact verification and reproduction}\label{app:verification-en}

The audited material is reproduced by two independent entry points.  From the
\file{Erdos506} directory, run
\begin{center}
\file{computation/verify_certified_results.py}
\end{center}
\begin{verbatim}
python computation/verify_certified_results.py
\end{verbatim}
This command rechecks the mathematical identities for $n\ge15$, the exact
constructions and classifications for $n\le8$, and
Corollary~\ref{thm:no-three-en}.  The six intervening layers are
reproduced separately by
\begin{verbatim}
cd computation/n9_to_14
python run_verification.py --jobs 8 --seconds 240
\end{verbatim}
as described below.  Neither entry point accepts random search, floating-point
feasibility, a timeout, or an unknown solver status as an exclusion certificate.
Both require CPython~3.13, matching the audited binary modules; SymPy performs exact polynomial algebra, and the
project-local OR-Tools version 9.15.6755 is pinned and loaded only through
\file{computation/runtime_dependencies.py}.

\subsection{The mathematical range}

\begin{longtable}{P{6.4cm}P{8.1cm}}
\toprule File & Independently checked assertion\\ \midrule
\file{computation/n_ge_17/verify_n_ge_17.py} &
expands the coefficient differences in
Lemma~\ref{lem:bojanowski-local-en}, the two parity polynomials, the endpoint
margins, and the largest-block algebra, all over the rationals\\
\file{computation/n16/verify_n16.py} &
expands the two identities in the $K=7,8$ branches, checks all residual
coefficients, the multiplicities forced at equality, the sums of pairwise
circle-intersection counts, and the
direct $K=9,\ldots,15$ bounds\\
\file{computation/n9_to_14/verify_largest_block_reduction.py} &
checks the closed-form line and circle bounds for every largest-block size
$K\ge8$ in the fifteen-point case\\
\file{computation/n15/verify_n15.py} &
checks the local residual values from Cuntz's fourteen-line table, the
positive-defect estimate, the coefficients of
\eqref{eq:n15-global-identity-en}, and the final congruence obstruction; it also
records the SHA-256 digest used to identify the cited source PDF\\
\bottomrule
\end{longtable}
These scripts verify algebra in proofs already given in the text.  The
fifteen-point check reads only the three published multiplicity vectors listed
in the proof; it does not choose labelled point subsets or enumerate orbits or
coordinates.

The arithmetic and finite-sum implications used for $n\ge15$ are also
formalized in
\begin{center}
\file{lean/Erdos506Nge15.lean}.
\end{center}
The file is compiled against the pinned Lean~4 and Mathlib toolchain.  It
contains no \texttt{sorry}, \texttt{admit}, custom axiom declaration, or unsafe
definition.  Its scope is deliberately explicit: the Euclidean and
real-projective incidence theorems cited in the text are hypotheses of the
corresponding Lean statements, while every algebraic consequence drawn from
those hypotheses is checked by the kernel.  Thus this file is a formalization
of the arithmetic-combinatorial layer of the proof, not a formal construction
of the underlying projective geometry.

The analogous arithmetic-combinatorial formalization of
Corollary~\ref{thm:no-three-en} is
\begin{center}
\file{lean/Erdos506NoThreeCollinear.lean}.
\end{center}
It kernel-checks the six- and nine-point packing deductions, the ten-point
local equality case, both large-order endpoint factors, the eight-point
congruence gap, and the final case split for every $n\ge4$.  The accompanying
\file{lean/compile_audit_no_three_collinear.json} records a successful warning-as-error build with
Lean~4.29.1 and Mathlib commit \texttt{5e932f97}; the source contains no proof
placeholder or added axiom.  As above, Euclidean incidence lemmas and exhaustive
finite-classification outputs remain explicit hypotheses rather than silently
postulated definitions.

\subsection{Nine through fourteen points}

The source files for these six layers are in
\begin{center}
\file{computation/n9_to_14/}.
\end{center}
From that directory run
\begin{verbatim}
python run_verification.py --jobs 8 --seconds 240
\end{verbatim}
The number of workers may be changed without changing the models.  Every
individual solver uses a fixed random seed; a timeout has status
\texttt{UNKNOWN} and is retained.  A successful replay ends with
\file{certificates/manifest.json}, whose seven terminal checks are
\texttt{true} and whose status is \texttt{PASS}.  If a replay is interrupted
after all terminal certificates have been written, the same terminal and
SHA-256 audit can be rerun with \file{audit_manifest.py}.
The active lower-bound table contains $c(7)=11$.  All certificates in the
present package were regenerated after that correction; in particular, the
twelve-point recursion now has 622, rather than 617, nine-point child
count vectors.

The files are grouped by mathematical role as follows.
Some filenames retain implementation terminology: in those names,
\texttt{signature} means the global count vector defined in
Section~\ref{sec:small-uniform-en}, \texttt{profile} means the degrees attached
to point pairs, and \texttt{support} means an actual subset of labelled points.
These words are not additional mathematical notions.
\begin{longtable}{P{6.2cm}P{8.3cm}}
\toprule File & Role in the verification\\ \midrule
\file{signature_filters.py} & enumerates all global count vectors from
\eqref{eq:small-line-pairs-en}--\eqref{eq:small-triple-partition-en}, the
largest-block bounds, deletion averages, and the cited line-arrangement
inequalities\\
\file{projected_arrangement_model.py} & generates all numerical local
multiplicity types and their line/circle splits; the final source contains no
historical line-deletion obstruction table\\
\file{verify_local_type_filter.py} & checks exact unlabelled sums of the local
multiplicity vectors and writes \file{nN_01_signature_filter.json}\\
\file{verify_augmented_line_melchior.py} & checks first
\eqref{eq:augmented-line-summed-en}, then the sums over all points required by
\eqref{eq:augmented-line-point-en}; it writes certificates 02 and 03\\
\file{filter_by_classified_split_endpoints.py} & checks the consistency
identities linking point degrees and point-pair degrees in
\eqref{eq:small-category-totals-en}--\eqref{eq:small-pair-profile2-en};
the runner explicitly disables both optional finite catalogues\\
\file{pair_profile_generator.py},
\file{verify_pair_moment_filter.py} & generate all point-pair degree data and
impose the exact sums of intersection sizes and of their second binomial
coefficients, together with the disjoint-triple counts; they write
certificate 05\\
\file{block_intersection_rows.py},
\file{block_row_existence_filter.py} & at $n=10$, impose the three exact
triple strata around each represented block and the bound on matchings between
two blocks, reducing 21 count vectors to 12\\
\file{verify_recursive_inversion_rows.py} & maps each local multiplicity vector
to the full count vector one order lower, balances all point totals at each
level, and at the nine-point base selects the maximal lines and circles as
subsets of the labelled points\\
\file{verify_n10_final_radical_axis.py} & reconstructs the two matching
orbits, the unique surviving line-design orbit, and the five exact minors that
exclude the final ten-point count vector; it first reads certificate 07, records
its SHA-256 digest, and checks that the retained rows force the two five-point
circles to partition the point set\\
\file{verify_n13_n14_global_inequalities.py} & independently regenerates the
3962 and 8981 raw count vectors and checks the global exclusions and
\eqref{eq:n13-global-two-ineq-en}; it does not select labelled point subsets\\
\file{run_verification.py}, \file{audit_manifest.py} & execute the stages in
order and check all terminal statuses, the unique intermediate $n=10$ count
vector, file sizes, and SHA-256 digests\\
\bottomrule
\end{longtable}

The terminal certificate paths are
{\raggedright
\begin{description}[style=nextline,leftmargin=2.8em]
\item[$n=9$] \file{certificates/n9_04_exact_support.json}.
\item[$n=10$] \file{certificates/n10_07_recursive.json} and
  \file{certificates/n10_08_final_geometry.json}.
\item[$n=11$] \file{certificates/n11_06_recursive.json}.
\item[$n=12$] \file{certificates/n12_06_recursive.json}.
\item[$n=13,14$] \file{certificates/n13_n14_global_inequalities.json}.
\end{description}
\par}
The intermediate filenames have the numerical prefixes displayed in the
table and are consumed in that order.  Each stage after 01 records the path
and SHA-256 digest of its immediate input whenever it reads one.  The final
manifest records forty-two source and certificate files.  Thus a reader may
delete the \file{certificates} directory and regenerate the complete chain
from the listed source files.

In the audited four-worker replay, the full chain took 348.52 seconds of wall
time.  The largest labelled base model had 840 Boolean variables and 546
constraints, and the slowest individual recursive solve took 6.62 seconds.
No solve approached the 240-second limit.  These figures are reported only to
indicate scale; they are not assumptions in any certificate.

The reproduction uses CPython~3.13 because the audited binary modules are its
CPython~3.13 builds.  The pinned integer solver is OR-Tools 9.15.6755, loaded through
\file{computation/runtime_dependencies.py}.  SymPy is used only for the exact
ten-point polynomial calculations.  Floating-point feasibility, random
search, and a solver status other than explicit infeasibility are never used
to exclude a case.

\subsection{Six and seven points}

\file{computation/n4_to_8/n6/verify_construction.py} checks the six-point coordinates,
maximal lines, circle equations, and triple ownership exactly in
$\mathbb Q(\sqrt3)$.  The finite parameter check
\file{computation/n4_to_8/n6/verify_lower_bound.py} reproduces the last
case of the six-point lower-bound argument.

For seven points the reproduction order is:
\begin{enumerate}
\item \file{computation/n4_to_8/n7/classify_quad_packings.py} classifies the families of at least
five four-point circles.  It returns two classes with five circles and one class
each with six and seven circles.
\item \file{computation/n4_to_8/n7/classify_line_extensions.py} proves that only the six- and
seven-circle families can have $c(P)\le10$.  It checks that the unique
six-circle candidate has no ordinary line and that every seven-circle candidate
contains the common two-line core used in the proof.
\item \file{verify_q7_exclusion.py}
checks the positive-definite quadratic-form reduction in the text, while
\file{verify_extremal_configuration.py} verifies the eleven-circle construction
using exact Euclidean circle equations.
\item \file{verify_circle_count_spectrum.py} reruns the incidence classification
at circle counts eleven and twelve and checks the four symbolic reductions used
to verify the eleven-circle layer and prove nonexistence at twelve circles.
\end{enumerate}
The quadratic-form script asserts the two rational factorizations used in the
proof and records every nonzero denominator.  It performs no root approximation
and no search over coordinates.  This formulation is invariant under the affine
normalization: the former standard-circle normalization, which did not retain the
cross term and the second diagonal coefficient, is not used.
All paths in this paragraph are under
\file{computation/n4_to_8/n7/}.

\subsection{Eight points}

All paths below are under
\file{computation/n4_to_8/n8/}.
\begin{enumerate}
\item Lemma~\ref{lem:largest-en}, audited by
\file{computation/n9_to_14/verify_largest_block_reduction.py},
first excludes every line or circle containing at least five points.
The counting argument in the proof first excludes circle counts below twelve.
\file{classify_boundary_layers.py} then constructs the two exhaustive branches
for every remaining circle count $12,\ldots,16$: a four-point line is fixed up to relabeling,
or no four-point line exists.  Besides triple ownership and the intersection
bounds, the model uses only the orchard bound $\ell_3\le7$, the identity
$\ell_2=28-3\ell_3-6\ell_4$, and Melchior's inequality
$\ell_2\ge3+\ell_4$.
Full relabeling-orbit blocking proves completeness.  The certificate
\file{boundary_layer_classification.json} records respectively
$1,1,1,3,3$ classes for counts $12,\ldots,16$.
\item The geometric exclusions are independent of the combinatorial solver.
\file{pattern1_inversion_relation.py} and
\file{pattern1_inversion_circles.py} verify the exceptional fifteen-circle
inversion equations.  The two files
\file{no_four_line_inversion_relations.py} and
\file{no_four_line_key_equations.py} verify the exceptional sixteen-circle
equations.  The remaining classes are excluded in the text by the
radical-axis assertion in Lemma~\ref{lem:two-line-geometry-en}.
For the exceptional fifteen-circle class, inversion at the intersection of its
two prescribed lines gives axes on which
$0,1,2=(x_0,0),(x_1,0),(x_2,0)$ and
$4,5=(0,y_4),(0,y_5)$.  The Euclidean quadratic form is
$x^2+y^2+2kxy$, where $k^2\ne1$.  The collinearity and first circle equations
verified by the two scripts give
\[
R=x_0x_1-2x_0x_2+x_1x_2=0,\qquad x_0x_2=y_4y_5.
\]
If $x_0x_2=y_4^2$, then $y_4=y_5$, so this factor is nonzero.  The next two
circle equations therefore reduce, with
$A=x_0x_2+y_4^2$ and $B=y_4(x_0+x_2)$, to $A=kB$ and $kA=B$.
Thus $A=B=0$, whence $x_0+x_2=0$; substituting in $R=0$ gives
$2x_0^2=0$, contrary to distinctness.  This is the complete geometric
consequence of the symbolic output used for that class.
\item Finally, \file{verify_construction.py} verifies the rational
seventeen-circle construction by exact determinants and checks ownership of all
$56$ triples.
\end{enumerate}

\subsection{The no-three-collinear variant}

The single file
\file{computation/no_three_collinear/verify_no_three_collinear.py} independently checks
all finite and algebraic assertions in Corollary~\ref{thm:no-three-en}.  It counts
the twenty circles of the inverted two-square construction, expands both
large-$n$ endpoint margins, enumerates the unique seven- and eight-point
four-subset packings used in the proof, classifies the $2$-$(6,3,2)$ design in
the ten-point equality branch, and verifies the coordinate eliminations.  The
enumerations are over finite labelled subsets and are then compared with the
full relabelling orbits.  All coordinate calculations are exact over the
rationals or in the symbolic polynomial ring.

\begin{remark}[Independent replay]
The classifiers construct their constraints from the full labelled sets of
triples, four-subsets, and five-subsets; the stored JSON files are outputs, not
axioms.  Orbit blocking is checked by applying all point permutations.  Therefore a
reader can delete the generated JSON files and reproduce them from the scripts.
The dependency loader checks the pinned solver version.  The classifiers use one
search worker where isomorphism completeness matters and treat only an exhaustive
terminal status as a certificate.
\end{remark}

\end{document}